\documentclass[11pt,twoside, leqno]{article}

\usepackage{amssymb}
\usepackage{amsmath}
\usepackage{mathrsfs}
\usepackage{amsthm}
\usepackage{amsfonts}
\usepackage{txfonts}
\usepackage{enumerate}
\usepackage{esint}
\usepackage{latexsym,amssymb}
\usepackage{color}
\usepackage{hyperref}

\allowdisplaybreaks

 \usepackage[pdftex]{graphicx}

\usepackage{geometry,color,mathtools,fixltx2e,microtype,lastpage,latexsym,dsfont, amsmath,amsfonts,amssymb,amsthm,mathrsfs}

\newcommand*\Laplace{\mathop{}\!\mathbin\bigtriangleup}

\usepackage{tikz}
  \usetikzlibrary{shapes,shadows}
  \tikzstyle{abstractbox} = [draw=black, fill=white, rectangle, 
  inner sep=10pt, style=rounded corners, drop shadow={fill=black,
  opacity=1}]
  \tikzstyle{abstracttitle} =[fill=white]

  \newcommand{\boxabstract}[2][fill=white]{
    \begin{center}
      \begin{tikzpicture}
        \node [abstractbox, #1] (box)                                          
        {\begin{minipage}{0.80\linewidth}
            \footnotesize #2
          \end{minipage}};
        \node[abstracttitle, right=10pt] at (box.north west) {Abstract};
      \end{tikzpicture}
    \end{center}
  }

\numberwithin{equation}{section}

\theoremstyle{definition}

\newtheorem{rem}{Remark}[section]
\newtheorem{lem}{Lemma}[section]

\newtheorem{thm}{Theorem}[section]
\newtheorem{prop}{Proposition}[section]

\def\R{\mathbb{R}}
\def\P{\mathbb{P}}

\def\N{\mathbb{N}}

\def\H{\mathfrak{H}}

 \def\E{\mathbb{E}}

\def\to{\rightarrow}

\def\H{\mathfrak{H}}

\def\Cov{\text{Cov}}
\def\Var{\text{Var}}

\begin{document}

\arraycolsep=1pt
\allowdisplaybreaks

\title{\bf\Large
A Peccati-Tudor type theorem for Rademacher chaoses   
\footnotetext {\hspace{-0.35cm}
2010 {\it Mathematics Subject Classification}. Primary: 60F05, 60B12;
Secondary: 47N30.
\endgraf {\it Key words and phrases}.
Fourth moment theorem; Rademacher chaos; Stein's method; exchangeable pairs;  spectral decomposition; maximal influence.
}}
\author{Guangqu Zheng\footnote{Email: guangqu.zheng@uni.lu}}
\date{\it\small 6 Avenue de la fonte, Maison du Nombre, \\
Universit\'e du Luxembourg, Esch-sur-Alzette, L4364, Luxembourg}
 \maketitle

\vspace{-0.6cm}

\boxabstract{In this article, we prove that in the Rademacher setting, a random vector with chaotic components is close in distribution to a centred Gaussian vector, if both the {\it maximal influence}  of the associated kernel and the fourth cumulant of each component is small. In particular, we recover the univariate case recently established in D\"obler and Krokowski (2017).  

\qquad Our main strategy consists in  a novel adaption  of the  exchangeable pairs couplings initiated in Nourdin and Zheng (2017), as well as its combination with   estimates via chaos decomposition.
    }


\section{Introduction}  

\subsection{Motivation}

Nualart and Peccati's  fourth moment theorem  states that a normalised sequence of fixed-order multiple Wiener-It\^o integrals associated to a Brownian motion converges in law to the standard Gaussian if and only if the corresponding fourth moment converges to $3$.  It was proved in \cite{FMT} using the Dambis-Dubins-Schwartz  random-time change technique.   Soon after the appearance of \cite{FMT}, several extensions have been made, among which the paper  \cite{PTudor}  by Peccati and Tudor provided a significant multivariate extension using the same techique.  Roughly speaking, a sequence of chaotic random vectors on the Wiener space converges in distribution to a centred Gaussian vector with matched covariance matrix if and only if the asymptotic normality holds true for each component.  Note that the necessary condition boils down to the convergence of the fourth moments due to the fourth moment theorem of Nualart and Peccati.

In 2009,  Nourdin and Peccati \cite{NP08}    combined the Malliavin calculus and Stein's method of normal approximation so as to literally create a new field of research, known as the {\it Malliavin-Stein approach}.  One of its many highlights is the obtention of the (quantitative) fourth moment theorem in the total-variation distance. Here is the  bound quoted from the monograph \cite{bluebook}:  given a normalised $q$-th Wiener-It\^o integral $F$ associated to a   Brownian motion, one has 
\begin{align*} 
 d_{\rm TV}(F, Z) := \sup_{A\in\mathscr{B}(\R)} \Big\vert \P\big( F\in A \big) - \P\big( Z\in A \big) \Big\vert  \leq  \frac{2}{\sqrt{3}} \sqrt{  \frac{q-1}{q} \big( \E[ F^4] - 3 \big) } \,\,, 
\end{align*}
where $Z$ is a standard Gaussian random variable and $\mathscr{B}(\R)$ denotes the Borel $\sigma$-algebra on $\R$.  As an immediate consequence, the fourth moment theorem of Nualart and Peccati follows.

The success of the Malliavin-Stein approach stems from the integration by parts on both sides, namely, the Stein's lemma within the Stein's method and the duality relation between Malliavin derivative and Skorohod divergence on a Gaussian space, see the monograph \cite{bluebook} for a comprehensive treatment.  The only ingredients required from the Stein's method are the    Stein's lemma, Stein's equation and the regularity properties of the Stein's solution, while ``exchangeable pairs",  another fundamental tool and notable cornerstone of Stein's method, had not been touched until the recent investigation \cite{NZ17} made by Nourdin and Zheng. They constructed infinitely many exchangeable pairs of Brownian motions and combined them with E. Meckes' abstract results \cite{Meckes_diss, Meckes2009} on exchangeable pairs  to recover the quantitative fourth moment theorem  on a Gaussian space in any dimension. Such an elementary  strategy was soon adapted by D\"obler, Vidotto and Zheng in \cite{DVZ17} for their investigation on the Poisson space, and they were able to obtain  the quantitative fourth moment theorem  in any dimension.  In fact, the univariate fourth moment theorem on the Poisson space was established earlier in \cite{DP17} under some {\it integrability} assumptions involving the difference operator, which are partially due to the inherent discreteness of the Poisson space.   Remarkably, the authors of \cite{DVZ17} were able to obtain the exact fourth moment theorem under the {\it weakest possible} assumption of finite fourth moment. This illustrates the power of   the elementary exchangeable pairs approach. 

\bigskip

In this work, under {\it suitable} assumptions,   we establish a Peccati-Tudor type theorem in the Rademacher setting using the elementary exchangeable pairs approach.

\subsection{Main result}
  
  We first fix a rich probability space $\big(\Omega, \mathcal{F}, \P \big)$, on which our random objects are defined. Let $\E$ be the associated expectation operator. 
  
We write $\N : = \{1, 2, \ldots \}$ and denote by $\mathbf{X}$ a sequence of independent Rademacher random variables $(X_k,k\in\N)$ such that $\P\big( X_k = 1\big) = p_k = 1 - q_k = 1 - \P\big( X_k = - 1 \big)\in(0,1)$.   We call it   the {\it symmetric} case, whenever $p_k  = 1/2$ for each $k\in\N$;  otherwise, we call it the {\it general} case.  We write $\mathbf{Y} = \big( Y_k, k\in\N \big)$ for the {\it normalised version} of $\mathbf{X}$, that is, 
\begin{align} \label{N-Version}
Y_k = \frac{X_k - p_k + q_k}{2\sqrt{p_kq_k}} \,\,,\quad k\in\N \, .
\end{align}
We write  $\H = \ell^2(\N)$, equipped with usual $\ell^2$-norm and   for $p\in\N$, $\H^{\otimes p}$ means the $p$-th tensor product of $\H$ and $\H^{\odot p}$  its symmetric subspace. We denote $\H^{\odot p}_0 : = \big\{ f\in \H^{\odot p}\,:\, f\vert_{\Laplace_p^c} = 0 \big\}$ with $\Laplace_p = \big\{ (i_1, \ldots, i_p)\in\N^p: i_k \neq i_j$ for different $k,j \big\}$. Clearly, $\H^{\odot 0}_0 = \H^{\otimes 0} = \R$ and $\H^{\odot 1}_0 = \H$.

 Let $f\in\H^{\odot d}_0$ with $d\in\N$ and $\Xi = (\xi_k, k\in\N)$ be a generic sequence of independent normalised random variables. We define the following   homogeneous sum with order $d$, based on the kernel $f$,  by setting,
\begin{align}\label{Q_d}
 Q_d(f; \Xi) : = \sum_{i_1, \ldots, i_d\in \N} f(i_1, \dots, i_d) \xi_{i_1}\cdots \xi_{i_d}
\end{align}
and in particular,  $Q_d(f; \mathbf{Y})$ is called the (discrete) multiple  integral of $f$.  We write  $\mathcal{C}_d = \big\{ Q_d(f; \mathbf{Y})\,: \, f\in\H^{\odot d}_0 \big\}$ and call it the $d$-th   Rademacher chaos, and as a convention, we put $\mathcal{C}_0 = \R$.    In case of no ambiguity, we will simple write $Q_d(f)$ for $Q_d(f; \mathbf{Y})$.

Let us introduce an important notion before we state our main result: for a given kernel $f\in\H^{\odot d}_0$, we denote by $\mathcal{M}(f)$ the {\it maximal influence} of $f$, namely
\begin{align}\label{m-inf}
\mathcal{M}(f) : = \sup_{k\in\N} \sum_{i_1, \ldots, i_{d-1}\in\N} f(i_1, \ldots, i_{d-1}, k)^2 \quad\text{for $d\geq 2$ \quad and}\quad \mathcal{M}(f) : = \sup_{k\in\N}  f( k)^2 \quad\text{for $d = 1$.}
\end{align}
This notion is adapted from the boolean analysis (see {\it e.g.} \cite{Boolean}), in which the class of low-influence functions is often what is interesting or necessary in practice. It is also closely related to the invariance principle established in \cite{MOO10} and the universality phenomenon  of Gaussian Wiener chaos \cite{NPR_aop}. See also Section 4 for more details.

\bigskip

In this work, we are mainly concerned with random variables in a Rademacher chaos and random vectors with components in Rademacher chaoses. More precisely,  we establish the following result.

  \begin{thm}\label{PT-Rad}  

Fix integers $d\geq 2$ and $1\leq q_1\leq \ldots \leq q_d$, and   consider  the sequence of random vectors
 $$
 F^{(n)} = (F^{(n)}_{1}, \ldots, F^{(n)}_d)^T :=  \big( Q_{q_1}(f_{1,n}),\ldots, Q_{q_d}(f_{d,n}) \big)^T
 $$ 
with kernels $f_{j,n}$ in $\H_0^{\odot q_j}$ for each $n\in\N$,$j\in\{1,\ldots,d\}$.  Assume that the covariance matrix $\Sigma_n$ of $ F^{(n)} $ converges in Hilbert-Schmidt norm to a nonnegative definite  symmetric matrix $\Sigma = \big( \Sigma_{i,j}, 1\leq i,j\leq d\big)$, as $n\to+\infty$.   Suppose that the following condition   holds:
 \begin{align*}
  \lim_{n\to+\infty} \sum_{j=1}^d   \mathcal{M}(f_{j,n})   = 0 \,\, . 
  \end{align*}
If for each $j\in\{1, \ldots, d\}$, $\E\Big[ \big(F^{(n)}_j\big)^4 \Big]$ converges to $3\Sigma_{j,j}^2$, as $n\to+\infty$, then  $ F^{(n)} $  converges in distribution to $Z\sim \mathcal{N}(0, \Sigma)$, as $n\to+\infty$. 
 \end{thm}

\bigskip

The above theorem is analogous to the Peccati-Tudor theorem on a Gaussian space \cite{PTudor}, so we call it a Peccati-Tudor type theorem, which explains our title.    One of the main tools we need for the proof is the following ingredient from Stein's method of exchangeable pairs.  As one will see easily, we can obtain a quantitative version of Theorem \ref{PT-Rad}, which will be an analogue to \cite[Theorem 1.7]{DVZ17} and left for interested readers.

Recall first that two random variables $W$ and  $W'$, defined on a common probability space, are said to form an exchangeable pair, if $(W, W')$ has the same distribution as  $(W', W)$.

\begin{prop}[Proposition 3.5 in \cite{DVZ17}]\label{Meckes09} 
For each $t > 0$, let $(F,F_t)$ be an exchangeable pair of centred  $d$-dimensional random vectors defined on a common probability space. Let $\mathscr{G}$ be a $\sigma$-algebra that contains $\sigma\{F\}$. Assume that $\Lambda\in\R^{d\times d}$ is an invertible deterministic matrix and $\Sigma$ is a symmetric, non-negative definite deterministic matrix such that 
\begin{enumerate}
\item[(a)]
${\displaystyle \lim_{t\downarrow 0} \frac{1}{t}\,  \E\big[ F_t - F |\mathscr{G} \big] =  - \Lambda F}$ in $L^1(\Omega)$,
\item[(b)]
${\displaystyle
 \lim_{t\downarrow 0} \frac1t\, \E\big[ ( F_t - F)(F_t- F)^T |\mathscr{G} \big] =  2\Lambda \Sigma + S }$ in $L^1(\Omega, \| \cdot \| _{\rm H.S.})$ for some   matrix $S=S(F)$, and with $\| \cdot \| _{\text{H.S.}}$ the Hilbert-Schmidt norm,
 \item[(c)]  for each $i\in\{1,\ldots, d\}$, there exists some real number $\rho_i(F)$ such that
${\displaystyle
\lim_{t\downarrow 0}   \frac{1}{t}\, \E\big[ ( F_{i,t} - F_i)^4 \big] =  \rho_i(F)}
$, 
where $F_{i,t}$ (resp. $F_i$) stands for the $i$-th coordinate of $F_t$ (resp. $F$).
\end{enumerate}
Then,   for $g\in C^3(\R^d)$ such that $g(F), g(Z)\in L^1(\P)$, we have, with $Z\sim \mathcal{N}(0,\Sigma)$,
\begin{eqnarray*}
&&\big| \E  [ g(F) ] - \E [g(Z)  ] \big| \\
&\leq&  \frac{ \| \Lambda^{-1} \| _{\text{op}}  \sqrt{d}\, M_2(g)  }{4} \E \left[\,\,  \sqrt{ \sum_{i,j=1}^d S_{i,j}^2 } \,\, \right] + \frac{\sqrt{d}  M_3(g) \| \Lambda^{-1} \| _{\text{op}}}{18} \sqrt{\sum_{i=1}^d 2\Lambda_{i,i} \Sigma_{i,i} + \E[S_{i,i} ] }\sqrt{\sum_{i=1}^d\rho_i(F)}\,,
\end{eqnarray*}
where $M_k(g) : = \sup_{x\in\R^d} \big\| D^k g(x) \big\|_\text{op}$ with $\| \cdot \| _{\text{op}}$ the operator norm.
                 \end{prop}
The rest of this paper is organised as follows: Section 1.3 is devoted to a brief overview of related results and we sketch our strategy of proving Theorem \ref{PT-Rad} in Section 1.4; in Section 2,  we provide preliminary knowledge on Rademacher chaos and a crucial {\it exchangeable pairs coupling}.  The proof of our main result will be given in Section 3 and some discussion about universality around Rademacher chaos will be presented in Section 4.

\subsection{A brief overview of literature}

Soon after  the appearance of \cite{NP08}, Nourdin, Peccati and Reinert combined Stein's method and a discrete version of Malliavin calculus to study the Gaussian approximation of Rademacher functionals in the symmetric case.  This analysis is known as the {\it discrete Malliavin-Stein approach}.  It has been generalised by the authors of \cite{KRT16, KRT_aop} not only in the multivariate setting but also in the general case where functionals involving non-symmetric, non-homogeneous Rademacher random variables were investigated.  Recently,  D\"obler and Krokowski \cite{DK17} gave the following {\it fourth-moment-influence} bound and pointed out that it is {\it optimal} in the sense that there are examples, in which the fourth moment condition alone would not guarantee the asymptotic normality. 
 
\begin{thm}[Theorem 1.1 in \cite{DK17}] \label{DK17-thm}  Fix  $p\in\N$ and $f\in\H^{\odot p}_0$ satisfying $p! \| f \| ^2_{\H^{\otimes p}} = 1$. Let $Z$ be a standard Gaussian   and $F =  Q_p(f; \mathbf{Y}  )\in L^4(\P)$, then we have the following bound in Wasserstein distance:
\[
 d_{\rm W}\big( F, Z \big) : = \sup_{\| h' \| _\infty \leq 1} \Big\vert \E\big[ h ( F  ) - h(Z) \big] \Big\vert  \leq C_1 \sqrt{\big\vert \E[ F^4 ] - 3 \big\vert    } + C_2 \sqrt{\mathcal{M}(f)} \,\,,
\]
where   $C_1, C_2$ are two numerical constants.  This result echoes the remarkable de Jong's central limit theorem \cite{deJong90}.

\end{thm}

 Besides the aforementioned references, Krokowski \cite{Kai-Poisson} derived a multiplication  formula that generalises the one in \cite{NPR-rad}, and applying as well the Chen-Stein's method, he   studied the Poisson approximation of Rademacher functionals.   Independently, Privault and Torrisi \cite{Torrisi-P} also derived a multiplication  formula and moreover, they obtained a generalisation of the approximate chain rule from \cite{NPR-rad}, and applied them to study Gaussian and Poisson approximation of Rademacher functionals in the general case.  Concerning the normal approximation in \cite{NPR-rad} or \cite{Torrisi-P}, the authors were only able to obtain the bounds in some ``smooth-version" distance, due to regularity involving in their chain rules and Stein's solution.   In a follow-up work, Zheng \cite{GZ16} obtained a neater chain rule that requires {\it minimal} regularity (see \cite[Remark 2.3]{GZ16}), from which he obtained the bound in Wasserstein distance as well as an almost sure central limit theorem for Rademacher chaos.     It is worthy pointing out that without using any chain rule, the authors of \cite{KRT16, KRT_aop} used carefully a representation of the discrete Malliavin gradient and the fundamental theorem of calculus to deduce the Berry-Esseen bound for normal approximation. Using similar ideas,  D\"obler and Krokowski \cite{DK17}  also provided the Berry-Esseen bound for their {\it fourth-moment-influence} theorem, which is of the same order as the above  Wasserstein bound.

\subsection{Strategy of proving Theorem \ref{PT-Rad}}

 Stein's method of exchangeable pairs was first systematically  presented   in Charles Stein's 1986 monograph  \cite{Stein86}, which was subsequently developed and ramified by many authors. Concerning our work, we mention in particular  E. Meckes' dissertation \cite{Meckes_diss}, in which she developed an infinitesimal version of this method to obtain total-variation bound in normal approximation. This infinitesimal version of Stein's method of exchangeable pairs was later generalised in \cite{CM08, Meckes2009} for the multivariate normal approximation.

 As announced, Proposition \ref{Meckes09} is one of our main tools, and it can be seen as a generalisation of \cite{Meckes2009}.  To use it, we need  to construct a suitable family of random vectors $F_t$, $t\geq 0$ such that $(F_t, F)$ is exchangeable for each $t$ and satisfies several {\it asymptotic regression} conditions.  In fact, we will first construct a family of Rademacher sequences $\mathbf{X}^t$ such that $\big( \mathbf{X}^t, \mathbf{X} \big)$ is an exchangeable pair of $\{ \pm 1 \}^\N$-valued random variables for each $t\geq 0$. More precisely, let $\mathbf{X}'$ be an independent copy of $\mathbf{X}$ and $\Theta = (\theta_k, k\in\N)$ be a sequence of  \emph{i.i.d.} standard exponential random variables such that $\mathbf{X}$, $\mathbf{X}'$ and $\Theta$ are independent.  For each $t\in[0,+\infty)$, we define 
 \[
 X_k^t : = X_k \mathbf{1}_{(\theta_k \geq t)} + X'_k \mathbf{1}_{(\theta_k < t)} \,\, .
 \]
   It has been pointed out in \cite{KRT_aop} that  $\mathbf{X}^t$ has the same distribution as $\mathbf{X}$, see also Remark 3.4 in \cite{NPR-rad} for the symmetric case.  However, both of these two articles did not explicitly state the exchangeability of  $\mathbf{X}^t$ and  $\mathbf{X}$, which will be proved in Lemma \ref{exch}. Assuming this and writing  $F = \mathfrak{f}(\mathbf{X} )$ for some representative $\mathfrak{f}: \{ \pm 1 \}^\N \to \R^d$, we can set $F_t = \mathfrak{f}(\mathbf{X}^t )$. It is easy to see that the exchangeability can be passed to $(F,F_t)$ now.   If $F= \big( Q_{p_1}(f_1; \mathbf{Y}) , \ldots, Q_{p_d}(f_d; \mathbf{Y}) \big)$, then we can write  $F_t = \big( Q_{p_1}(f_1; \mathbf{Y}^t) , \ldots, Q_{p_d}(f_d; \mathbf{Y}^t) \big)$ with $\mathbf{Y}^t$ the normalised version of $\mathbf{X}^t$ in the sense of \eqref{N-Version}.   
  
  Moreover, this exchangeable pairs coupling fits well with the Mehler's formula, which gives a nice representation of the discrete {\it Ornstein-Uhlenbeck semigroup} $\big( P_t, t\geq 0\big)$: given $F\in L^2\big(\Omega, \sigma\{ \mathbf{X} \} , \P \big)$, we can first write $F = \mathfrak{f}(\mathbf{X} )$ for some $\mathfrak{f}: \{ \pm 1 \}^\N \to \R$, then the {\it Mehler formula} (\cite[Proposition 3.1]{KRT_aop}) states that 
 \begin{align} \label{Mehler}
 P_t F = \E\Big[ \mathfrak{f}\big( \mathbf{X}^t \big)  \,\vert \, \sigma\{ \mathbf{X} \} \Big] \, .
 \end{align}
 For $\xi\in \mathcal{C}_p$, as we will see in Section 2, $P_t\xi = e^{-pt}\xi$, then the asymptotic linear regression (a) in Proposition \ref{Meckes09}  follows easily, and with slightly more effort, the higher order regressions can also be obtained, see Proposition \ref{GAMMA_exch}. 
 
 Another important ingredient in our proof is  Ledoux spectral point-of-view for fourth moment theorem \cite{Ledoux2010}, which was later refined \emph{e.g.} in \cite{ACP14, CNPP16}.  Such a spectral viewpoint helps one get rid of some computational deadlock that is usually caused by  the complicated multiplication formula. In particular, our proof is motivated by some arguments in \cite{CNPP16}.

 As a byproduct of our strategy, we will provide a short proof of  Theorem \ref{DK17-thm} in the beginning of Section 3.  Some estimate from this proof will also be helpful for our multivariate case.

 \paragraph{Acknowledgement.} Part of this work was done during a visit at National University of Singapore.  I thank very much Professor Louis H. Y. Chen at NUS for his very generous support and kind hospitality.  The gratitude also goes to Professor Giovanni Peccati for sharing his alternative proof of Lemma 2.4 in \cite{DP17}, which motived our proof of Lemma \ref{DK-lemma}.

\section{Preliminaires}
Denote by $\sigma\{\mathbf{X}\}$ the $\sigma$-algebra generated by the sequence $\mathbf{X}$, and note that $\sigma\{ \mathbf{X} \} = \sigma\{ \mathbf{Y} \}$.  The   {\it Wiener-It\^o-Wash chaos decomposition} asserts that any random variable $F\in L^2\big(\Omega, \sigma\{\mathbf{X}\}, \P \big)$ admits a unique representation
\begin{align}\label{chaos}
F= \E[ F] + \sum_{p\geq 1} Q_p(f_p) \quad\text{with $f_p\in \H^{\odot p}_0$ for each $p\in\N$,}
\end{align}
where the above series converges in $L^2(\P)$. We denote by  $J_k(\cdot)$ the projection onto the $k$-th Rademacher chaos $\mathcal{C}_k$: for $F$ given in \eqref{chaos}, $J_p(F) = Q_p(f_p)$ for each $p\in\N$, and $J_0(F) = \E[F]$.  It is not difficult to check that for $f\in\H^{\odot p}_0$ and $g\in\H^{\odot q}_0$, it holds that
 \[
  \E\big[ Q_p(f) Q_q(g) \big] = \mathbf{1}_{\{ p=q\}} p! \langle f, g \rangle_{\H^{\otimes p}}  \, .
 \]
This is known as the {\it orthogonality} property of the multiple integrals.   One can refer to N. Privault's survey \cite{Privault-survey} for more details and relevant discrete Malliavin calculus. 

The authors of \cite{NPR-rad} established a multiplication formula for discrete multiple integral in the symmetric case:  given   $f\in\H^{\odot p}_0$  and   $g\in\H^{\odot q}_0$, one has
\begin{align}\label{product}
Q_p(f)  Q_q(g) = \sum_{r=0}^{p\wedge q} r! {p\choose r} {q\choose r} Q_{p+q-2r}\big( f\widetilde{\otimes}_r g\mathbf{1}_{\Laplace_{p+q-2r}} \big) \,\,,
\end{align}
where the $r$-{\it contraction} $f\otimes_r g$ of $f$ and $g$ is defined by 
\[
 (f\otimes_r g)\big(i_1, \ldots, i_{p-r}, j_1, \ldots, j_{q-r} \big) := \sum_{k_1, \ldots, k_r\in\N} f\big(i_1, \ldots, i_{p-r}, k_1, \ldots, k_r \big) \cdot g\big(j_1, \ldots, j_{q-r}, k_1, \ldots, k_r \big)
\]
 and $f\widetilde{\otimes}_r g$ is the canonical symmetrisation of $f\otimes_r g$, {\it i.e.} for any $h\in\H^{\otimes p}$,  $\widetilde{h}$ is given by 
 \[
 \widetilde{h}(i_1, \ldots, i_p)  = \frac{1}{p!} \sum_{\sigma\in\mathfrak{S}_p} h \big(i_{\sigma(1)}, \ldots, i_{\sigma(p)} \big)\,\,,
 \]
 with $\mathfrak{S}_p$ the permutation group over $\{ 1, \ldots, p \}$.    We follow the convention that $\widetilde{c} = c$ for each $c\in\R$. Note it is easy to deduce from the Cauchy-Schwarz inequality  that $\| \widetilde{h} \| _{\H^{\otimes p }} \leq \| h \| _{\H^{\otimes p}}$ for each $h\in \H^{\otimes p}$, then applying the above orthogonality property and mathematical induction gives us a weak form of the {\it hypercontractivity} property in the symmetric case, namely,  $\E\big[ \vert F \vert^r \big] < +\infty$ for any $F\in \mathcal{C}_p$, $p,r\in\N$.
   
   However, in the general case,  one can {\it not} even guarantee the existence of finite fourth moment of  a generic  multiple integral.  Such a phenomenon, due to the asymmetry, is also revealed in the corresponding multiplication formulae, see   Proposition 2.2 in \cite{Kai-Poisson} and Proposition 5.1 in \cite{Torrisi-P}.  As already pointed out in \cite{DK17},  given $F\in \mathcal{C}_p\cap L^4(\P)$,  one can {\it not} directly deduce from these multiplication formulae that $F^2$ admits a finite chaotic decomposition.  Adapting the induction arguments from the proof of \cite[Lemma 2.4]{DP17}, D\"obler and Krokowski gave the following positive result. 
   
   \begin{lem}[Lemma 2.3 in \cite{DK17}]\label{DK-lemma}  Let $F = Q_p(f)   \in L^4(\P)$ and $G = Q_q(g)\in  L^4(\P)$ for some $f\in \H^{\odot p}_0$ and $g\in \H^{\odot q}_0$. Then $FG\in L^2(\P)$ admits a finite chaos decomposition of the form 
   \[
   FG = \E[FG] +  \sum_{k=1}^{p+q-1} J_k(FG) + Q_{p+q}\big( f\widetilde{\otimes} g \mathbf{1}_{\Laplace_{p+q}} \big) \,\,.
   \]
In particular, if $Q_1(h)$ belongs to $L^4(\P)$ for some $h\in\H$, then 
\[
Q_1(h)^2 = \| h \| ^2 _\H + Q_1(w) + Q_2\big(h\widetilde{\otimes} h \mathbf{1}_{\Laplace_2}\big) \quad\text{with $w(k) = \frac{h(k)^2(q_k - p_k)}{\sqrt{p_kq_k}}, k\in\N$.}
\]   
({\it As this lemma is crucial for our work and for the sake of completeness, we provide in Section 3.3 another and direct proof  suggested by Giovanni Peccati.})   
   
   \end{lem}

\subsection{Ornstein-Uhlenbeck Structure and carr\'e du champs operator}

 Denote by $\text{dom}(L)$ the set of  those $F$ in \eqref{chaos} verifying 
\[
\sum_{p=1}^\infty  p^2 \E\big[ Q_p(f_p)^2 \big] = \sum_{p=1}^\infty p^2 p!  \| f_p  \| ^2_{\H^{\otimes p}} < +\infty \, .
\]
 For such a $F\in\text{dom}(L)$, we define $LF = - \sum_{p\geq 1} p Q_p(f_p)$. In particular, if $F\in \mathcal{C}_p$, $LF = - pF$. In other words, $- L$ has pure spectrum $\N\cup\{0\}$ and each eigenvalue $p\in\{0\}\cup\N$ corresponds to the eigenspace $\mathcal{C}_p$.
  And we call $L$ the {\it Ornstein-Uhlenbeck operator}, equipped with its domain $\text{dom}(L)$. 
  
   For $F, G\in\text{dom}(L)$ such that $FG\in\text{dom}(L)$, we define the {\it carr\'e du champs} operator $\Gamma(F,G)$ by setting
  \[
  \Gamma(F,G) : = \frac{1}{2} \big( L(FG) - FLG - GLF \big) \,\, .
  \]
 In particular, for $F, G$ as in Lemma \ref{DK-lemma}, one has $FG\in\text{dom}(L)$ and 
 \begin{align}\label{sp-dec}
  \Gamma(F,G) = \frac{1}{2}\big[ (p+q) +  L \big] \left( \sum_{k=0}^{p+q} J_k(FG)\right) = \frac{p+q}{2}\E[ FG] + \sum_{k=1}^{p+q-1} \frac{p+q- k}{2} J_k(FG) \, ,
 \end{align}
 and as a consequence of the orthogonality property, one deduces  that 
 \begin{align}
 \Var\big( \Gamma(F,G) \big) = \sum_{k=1}^{p+q-1}  \frac{(p+q- k)^2}{4} \Var\big(  J_k(FG) \big) \leq \max\{ p^2, q^2\} \sum_{k=1}^{p+q-1}  \Var\big(  J_k(FG) \big) \, ,\label{sp-eq}
 \end{align}
 which is all we need about the carr\'e du champs.
 
 \bigskip
 
For each $t\in[0, +\infty)$  and  $F$ as in \eqref{chaos}, we define
 $$ 
 P_tF := \E[F] + \sum_{p = 1}^\infty e^{-pt} Q_p(f_p) \,\,.
 $$
 $(P_t, t\geq 0 )$ is called the Ornstein-Uhlenbeck semigroup, which can be represented alternatively by the  Mehler formula \eqref{Mehler}. To verify \eqref{Mehler}, one can first consider $F = Q_p(f_p)$ in a Rademacher chaos with $f_p\in\H^{\odot p}_0$ having  finite support and then use the standard approximation argument.  Note that for $F\in\text{dom}(L)$, it is not difficult to check $t^{-1} (P_t F - F)$ converges in $L^2(\P)$ to $LF$, as $t\downarrow 0$.

 \subsection{Exchangeable pairs of Rademacher sequences}

\begin{lem}\label{exch} Let $\mathbf{X}^t $ and $\mathbf{X}$  be given as before, then $\big( \mathbf{X} , \mathbf{X}^t \big)$ has the same distribution as $\big( \mathbf{X}^t , \mathbf{X} \big)$. In particular, for any $f_j\in \H^{\odot p_j}_0$ with $p_j\in\N$, $j = 1, \ldots, d$,  
\begin{center} 
$\big( Q_{p_1}(f_1; \mathbf{Y}) , \ldots, Q_{p_d}(f_d; \mathbf{Y}) \big)   $ and $\big( Q_{p_1}(f_1; \mathbf{Y}^t) , \ldots, Q_{p_d}(f_d; \mathbf{Y}^t) \big)   $   
\end{center}
form an exchangeable pair, where $\mathbf{Y}^t$ stands for the normalised version of $\mathbf{X}^t$ in the sense of \eqref{N-Version}.

\end{lem}

\begin{proof}  Note first that $\mathbf{X}^t$ is a sequence of independent Rademacher random variables for each $t\in[0,+\infty)$. For each $k\in\N$, it is easy to check that 
\[
\P\big( X_k^t = -1, X_k = 1 \big) = \P\big( X_k^t = 1, X_k =  -1 \big) = (1- e^{-t} ) p_kq_k \,\, .
\]
 This gives us the exchangeability of $(X_k, X_k^t)$ for each $k\in\N$. Let $\mathbf{a} = (a_i, i\in\N), \mathbf{b} = (b_i, i\in\N)\in \{ \pm 1 \}^\N$, then using the independence within those two sequences 
$ \mathbf{X} , \mathbf{X}^t $, we obtain
\begin{align*}
\P\big(  \mathbf{X} = \mathbf{a} , \mathbf{X}^t  = \mathbf{b} \big) = \prod_{k\in\N} \P\big(  X_k =a_k , X^t_k  =b_k \big)  &=  \prod_{k\in\N} \P\big(  X_k =b_k , X^t_k  =a_k \big)  \quad \text{by exchangeability of $X_k$, $X_k^t$} \\
& = \P\big(  \mathbf{X} = \mathbf{b} , \mathbf{X}^t  = \mathbf{a} \big) \,\, .
 \end{align*}
This proves the exchangeability of $ \mathbf{X} , \mathbf{X}^t $. The rest follows from a standard approximation argument: it is clear that after truncation, (with $[N] : = \{1, \ldots, N\}$)
\[
\text{$\big( Q_{p_1}(f_1\mathbf{1}_{[N]^{p_1}}; \mathbf{Y}) , \ldots, Q_{p_d}(f_d\mathbf{1}_{[N]^{p_d}}; \mathbf{Y}) \big)   $ and $\big( Q_{p_1}(f_1\mathbf{1}_{[N]^{p_1}}; \mathbf{Y}^t) , \ldots, Q_{p_d}(f_d\mathbf{1}_{[N]^{p_d}}; \mathbf{Y}^t) \big)   $   }
\]
form an exchangeable pair; letting $N\to+\infty$ and keeping in mind that the exchangeability is preserved in limit, we get the desired result. \qedhere

\end{proof}

The following result brings more connections between our exchangeable pairs and Ornstein-Uhlenbeck operator.

\begin{prop}\label{GAMMA_exch}
 Let $F= Q_p(f ;\mathbf{Y})\in L^4(\P)$ for some $f\in \H^{\odot p}_0$ and     define $F_t = Q_p(f ;\mathbf{Y}^t)$. Then, $(F, F_t)$ is  an exchangeable pair for each $t\in\R_+$. Moreover,
\begin{itemize}
 \item[\rm (a)] ${\displaystyle  \lim_{t\downarrow 0}  \frac{1}{t} \E\big[ F_t - F \vert \sigma\{ \mathbf{X} \} \big] =  L F = -pF}$ in $L^4(\P)$.
 
 \item[\rm (b)] If $G = Q_q(g ; \mathbf{Y})\in L^4(\P)$ and $G_t =Q_q(g ; \mathbf{Y}^t)$ for  some $g\in   \H^{\odot q}_0$, 
 
  then we have     ${\displaystyle  \lim_{t\downarrow 0}  \frac{1}{t} \E\big[ (F_t - F)(G_t - G) \vert \sigma\{ \mathbf{X}\} \big] =  2 \Gamma(F,G)}$, with the convergence    in $L^2(\P)$. 
 
 \item[\rm (c)]   ${\displaystyle  \lim_{t\downarrow 0}  \frac{1}{t} \E\big[ (F_t - F)^4 \big] = -4p \,\E[F^4] + 12 \, \E\big[ F^2 \Gamma(F,F) \big] \geq 0}$.

\end{itemize}

\end{prop}
 
\begin{proof}   By the  Mehler formula \eqref{Mehler}, we have 
\begin{align*}
  \frac{1}{t} \E\big[ F_t - F \vert \sigma\{ \mathbf{X}\} \big]    =  \frac{P_t(F) - F}{t}  = \frac{e^{-pt} - 1}{t} F \, ,
  \end{align*}
  converges  in $L^4(\P)$ to  $-pF = LF$, as $t\downarrow 0$.   As a consequence of Lemma \ref{DK-lemma}, $FG$ has a finite chaos expansion of the form
$
FG = \E[FG] +  \sum_{k=1}^{p+q} Q_k\big( h_k ; \mathbf{Y} \big) 
$
 for some   $h_k\in \H^{\odot k}_0$.  Therefore,  $F_tG_t = \E[FG]  +   \sum_{k=1}^{p+q} Q_k\big(h_k; \mathbf{Y}^t \big)$, implying
$$ \frac{1}{t} \E\big[ F_tG_t -FG \vert \sigma\{ \mathbf{X}\} \big]  =  \sum_{k=1}^{p+q}  \frac{1}{t} \E\big[ Q_k\big(h_k; \mathbf{Y}^t \big) -   Q_k\big(h_k; \mathbf{Y} \big) \vert \sigma\{\mathbf{X}\} \big]  $$
converges     in $L^2(\P)$ to  $\sum_{k=1}^{p+q} -k \, J_k(FG) = L(FG)$, as $t\downarrow 0$.  Hence, we infer that  in  $L^2(\P)$ and as $t\downarrow 0$,
\begin{eqnarray*}
  \frac{1}{t} \E\big[ (F_t - F)(G_t - G) \vert \sigma\{ \mathbf{X}\} \big]    &=  & \frac{1}{t} \E\big[ F_tG_t -FG \vert \sigma\{ \mathbf{X}\} \big]  -  F\frac{\E [ G_t -G \vert \sigma\{ \mathbf{X}\}  ] }{t}  -  G\frac{\E [ F_t -F \vert \sigma\{ \mathbf{X}\}  ] }{t} \\
  &\to& L(FG) - FLG - GLF = 2\, \Gamma(F,G) \, .
 \end{eqnarray*}
 Since the pair $(F,F_t)$ is exchangeable, we can write
    \begin{eqnarray}
  \E\big[ (F_t - F)^4 \big] & =&   \E\big[ F_t^4 + F^4 - 4 F_t^3F - 4 F^3 F_t + 6 F_t^2 F^2   \big]  \notag\\
  & = &2 \E[ F^4] - 8 \E\big[ F^3 F_t \big] +  6 \E\big[ F^2 F_t^2    \big] \qquad \big(\text{by exchangeability of $(F, F_t)$}\big) \notag\\
  & = &4 \E\big[ F^3(F_t - F) \big]  +6 \E\big[ F^2   (F_t - F)^2 \big] \quad\text{(after rearrangement)}\notag \\
  & = &4 \E\big[ F^3 \E[ F_t - F \vert \sigma\{ \mathbf{X} \} ] \big]  +6 \E\big[ F^2 \E[   (F_t - F)^2 \vert \sigma\{ \mathbf{X} \} ] \big].\notag
    \end{eqnarray}   
so (c) follows immediately from  (a),(b) and the fact that $F\in L^4(\P)$.             \end{proof}

\section{Proofs}

 We begin with the following lemma, whose proof is postponed to Section 3.3.

\begin{lem}\label{tech1}  Given $F = Q_p( f)$ with $f\in\H^{\odot p}_0$ and $G = Q_q(g)$ with $g\in\H^{\odot q}_0$,  we assume that $F, G\in L^4(\P)$. Then we have the following estimates:
\begin{align}\label{lem-1}
 \sum_{k=1}^{p+q-1}\Var\big( J_k(FG) \big) \leq \E\big[F^2G^2\big] - 2\E[ FG]^2 - \Var(F) \Var(G) +  (p+q)! \big\| f \widetilde{\otimes} g \mathbf{1}_{\Laplace_{p+q}^c} \big\| ^2_{\H^{\otimes p+q}} \, ,\end{align}
and in particular, 
\begin{align}\label{lem-2}
 \max\left\{\,  \sum_{k=1}^{2p-1}\Var\big( J_k(F^2) \big) , p!^2\,  \sum_{r=1}^{p-1} {p\choose r}^2 \big\| f\otimes_r f \big\| ^2_{\H^{\otimes 2p-2r} } \, \right\} \leq  \E\big[F^4\big] - 3\E[ F^2]^2  + (2p)! \big\| f \widetilde{\otimes} f \mathbf{1}_{\Laplace_{2p}^c} \big\| ^2_{\H^{\otimes 2p}} \, ,
\end{align}
with
\begin{align}\label{lem-3}
  \big\| f \widetilde{\otimes} g \mathbf{1}_{\Laplace_{p+q}^c} \big\| ^2_{\H^{\otimes p+q}} \leq \sum_{r=1}^{p\wedge q} r! {p\choose r}{q\choose r}        \min\Big\{ \,  \| f \|^2_{\H^{\otimes p}}   \mathcal{M}(g) ,   \| g \|^2_{\H^{\otimes q}}   \mathcal{M}(f)  \Big\} \, .
\end{align}
(As a convention, we put ${\displaystyle \sum_{r=1}^0 = 0}$.)
\end{lem}

Before we prove our multivariate limit theorem,  we will give a short proof of the univariate case in Wasserstein distance, using our exchangeable pairs coupling.

\subsection{Alternative proof of Theorem \ref{DK17-thm}}

We need the following result, which is the univariate analogue of Proposition \ref{Meckes09}.

 \begin{prop}\label{Meckes06}  Let $F$ and a family of  real random variables $(F_t)_{t\geq 0}$   be defined on a common probability space $(\Omega,\mathcal{F},\P)$ such that $F_t\overset{law}{=} F$ for every $t\geq 0$. Assume that $F\in L^4(\Omega, \mathscr{G}, \P)$ for some $\sigma$-algebra $\mathscr{G}\subset\mathcal{F}$  and that in $L^1(\P)$,
\begin{enumerate}
\item[(a)] ${\displaystyle \lim_{t\downarrow 0} \frac{1}{t}\, \E\big[ F_t -F  |\mathscr{G}\big] = -\lambda\,F}$ for some $\lambda > 0$,
\item[(b)] ${\displaystyle   \lim_{t\downarrow 0} \frac{1}{t} \,\E\big[ (Y_t-Y)^2|\mathscr{G}\big] = (2\lambda+S)\Var(F) }$ for some random variable $S$;
\item[(c)] and  ${\displaystyle  \lim_{t\downarrow 0} \frac{1}{t}\,\E\big[(F_t-F)^4\big]= \rho(F)\Var(F)^2}$ for some  $\rho(F) \geq 0$.
\end{enumerate}
Then, with $Z\sim \mathcal{N}\big(0, \Var(F)\big) $, we have  
$$ d_\text{W}(F,Z) \leq   \frac{\sqrt{\text{Var}(F)}}{\lambda\sqrt{2\pi}} \E\big[|S|\big] + \frac{ \sqrt{  ( 2\lambda+\E[S]) \Var(F)       }}{3\lambda}   \sqrt{\rho(F)} \,.   $$
For the proof, one can refer to \cite[Proposition 3.3]{DVZ17}. One may also want to refer to Theorem 3.5 of \cite{NPR-rad} for a different coupling bound. 
\end{prop}

\bigskip

Now given $F = Q_p\big(f ; \mathbf{Y} \big)\in L^4(\P)$ (with $\E\big[ F^2 \big]= 1$),  we can get by using \eqref{sp-eq} and \eqref{lem-2} that 
\begin{align} 
 \Var\Big( p^{-1} \Gamma(F,F) \Big) & \leq  \sum_{k=1}^{2p-1} \Var\Big( J_k(F^2)\Big)  \leq   \E[F^4] - 3\E[F^2]^2  +  (2p)!  \big\| f \widetilde{\otimes} f \mathbf{1}_{\Laplace_{2p}^c} \big\| ^2_{\H^{\otimes 2p}}\notag  \\
 &\leq    \E[F^4] - 3\E[F^2]^2     +  \gamma_p \E[F^2]   \mathcal{M}(f) \,  \quad\text{with $\gamma_p: =  \frac{(2p)!}{p!}  \sum_{r=1}^p r! {p\choose r}^2$}. \label{use1}
\end{align}
Also using the chaos expansion of $F^2$ and $\Gamma(F,F)$ as well as the orthogonality property, we have 
\begin{align*}
&\qquad 3 \E\big[ F^2 \Gamma(F,F) \big] - p \E[F^4]  = 3 \E\Big[ F^2\big(  \Gamma(F,F) - p \big) \Big] - p\Big( \E[F^4] - 3  \Big)  \\
& = 3 \E\left[ \,\, \left( \sum_{k=0}^{2p} J_k(F^2) \right) \left( \sum_{k=1}^{2p-1}  \frac{2p-k}{2} J_k(F^2) \right)  \,\, \right] - p\Big( \E[F^4] - 3  \Big) \leq 3p \sum_{k=1}^{2p-1}  \Var\Big( J_k(F^2)\Big) - p\Big( \E[F^4] - 3  \Big) \, .
\end{align*}
It follows from \eqref{use1} that 
\begin{align}\label{use2}
3 \E\big[ F^2 \Gamma(F,F) \big] - p \E[F^4]  \leq  2p\Big( \E[F^4] - 3  \Big) + 3p \gamma_p  \mathcal{M}(f) \,. 
\end{align}
Now define $F_t = Q_p\big(f; \mathbf{Y}^t \big)$ for each $t\in[0,+\infty)$, then by Proposition \ref{GAMMA_exch}, $(F_t, F)$ is an exchangeable pair satisfying the conditions in Proposition \ref{Meckes06} with $\mathscr{G} = \sigma\{\mathbf{X} \}$,  $\lambda = p$, $S= 2 \Gamma(F,F) - 2p$ and $\rho(F) = -4p \,\E[F^4] + 12 \, \E\big[ F^2 \Gamma(F,F) \big]$. Therefore, 
\begin{align*}
d_{\rm W}(F, N)& \leq  \frac{1}{p\sqrt{2\pi}} \E\big[ \vert 2 \Gamma(F,F) - 2p \vert \big] + \frac{\sqrt{2p}}{3p} \sqrt{-4p \,\E[F^4] + 12 \, \E\big[ F^2 \Gamma(F,F) \big]     } \\
&\leq   \frac{2}{\sqrt{2\pi}}  \sqrt{ \Var\Big( p^{-1}  \Gamma(F,F)   \Big) } +  \frac{\sqrt{2p}}{3p} \sqrt{-4p \,\E[F^4] + 12 \, \E\big[ F^2 \Gamma(F,F) \big]     } \quad\text{(since $\E[ \Gamma(F,F) ] = p$)} \\
&\leq   \sqrt{2/\pi} \sqrt{   \E[F^4] - 3    +  \gamma_p    \mathcal{M}(f)    } + \frac{2\sqrt{2}}{3} \sqrt{    2\big( \E[F^4] - 3 \big)    +  3 \gamma_p    \mathcal{M}(f)  } \\
&\leq  \big( \sqrt{2/\pi} + \frac{4}{3} \big) \sqrt{\vert \E[F^4] - 3  \vert } +   \big( \sqrt{2/\pi} + \frac{2\sqrt{6}}{3} \big) \sqrt{\gamma_p} \sqrt{\mathcal{M}(f)}
\end{align*}

This proves Theorem \ref{DK17-thm} with $C_1 =  \sqrt{2/\pi} + \dfrac{4}{3} $ and  $C_2 =  \big( \sqrt{2/\pi} + \dfrac{2\sqrt{6}}{3} \big) \sqrt{ \dfrac{(2p)!}{p!}  {\displaystyle \sum_{r=1}^p r! {p\choose r}^2} }$.

 \begin{rem} \begin{itemize}
 \item[(1)] For $F$ in the first Rademacher chaos, one can  directly prove  Theorem \ref{DK17-thm} without using the exchangeable pairs.  Indeed,  if $F = Q_1(h)\in L^4(\P)$ for some $h\in\H$ with $\| h \| _\H = 1$ and $Z\sim\mathcal{N}(0,1)$, then by \cite[Theorem 3.1]{GZ16},
 \[
 d_{\rm W}(F,Z) \leq \sqrt{\sum_{k=1}^\infty \frac{1}{p_kq_k} h(k)^4   }\, .
 \]
 By Lemma \ref{DK-lemma}, 
$
 F^2  = 1 + Q_1(w) + Q_2\big(h \otimes h \mathbf{1}_{\Laplace_2}\big)$
with $w(k) = \frac{h(k)^2(q_k - p_k)}{\sqrt{p_kq_k}}, k\in\N$. This implies 
\begin{align*}
\E\big[ F^4 \big] & = 1 + \sum_{k=1}^\infty h(k)^4 \frac{(q_k-p_k)^2}{p_kq_k} + 2 \| h\otimes h \| ^2_{\H^{\otimes 2}} - 2  \| h\otimes h \mathbf{1}_{\Laplace_2^c} \| ^2_{\H^{\otimes 2}} \\
& = 3 +  \sum_{k=1}^\infty h(k)^4 \frac{(q_k-p_k)^2}{p_kq_k} -  2 \sum_{k=1}^\infty h(k)^4 = 3 +  \sum_{k=1}^\infty h(k)^4 \frac{q_k^2 + p_k^2}{p_kq_k} -  4 \sum_{k=1}^\infty h(k)^4\,.
\end{align*}
Noticing  $p_k^2 + q_k^2 \geq 1/2$ for each $k\in\N$, we have 
\[
\frac{1}{2}\sum_{k=1}^\infty \frac{1}{p_kq_k} h(k)^4   \leq   4 \sum_{k=1}^\infty h(k)^4 + \E\big[ F^4 \big]  - 3 \leq 4 \mathcal{M}(h) + \E\big[ F^4 \big]  - 3\, .
\]
Hence, $ d_{\rm W}(F,Z) \leq \sqrt{2} \sqrt{\big\vert \E[F^4] - 3  \big\vert } + 2\sqrt{2} \sqrt{\mathcal{M}(h)}$. Moreover,  using the so-called second-order Poincar\'e inequality in \cite[Theorem 4.1]{KRT_aop}, we can have the Berry-Esseen bound
\[
d_{\rm Kol}\big( F, Z \big) := \sup_{z\in\R}\big\vert \P\big(F\leq z \big) -  \P\big(Z\leq z \big) \big\vert \leq 2 \sqrt{\sum_{k=1}^\infty \frac{1}{p_kq_k} h(k)^4   } \leq 2\sqrt{2} \sqrt{\big\vert \E[F^4] - 3  \big\vert } + 4\sqrt{2} \sqrt{\mathcal{M}(h)} \, .
\]

\item[(2)] Continuing the discussion in previous point and assuming $p_k = p = 1- q = 1 - q_k$ for each $k$, we have 
\begin{align}\label{007}
\E\big[ F^4  \big] - 3  =    \frac{p^2 + q^4 - 4pq}{pq} \sum_{k=1}^\infty h(k)^4\,\,.
 \end{align}
If $p\in(0,1)\setminus \{ \frac{1}{2} \pm \frac{1}{2\sqrt{3}} \}$, then we have the exact fourth moment bounds:
 \[
  d_{\rm W}(F,Z) \leq \sqrt{ \frac{1}{pq} \sum_{k=1}^\infty h(k)^4   }\leq   \left(   \frac{\E[ F^4 ] - 3}{p^2 + q^2 - 4pq} \right)^{1/2} \, \quad\text{and}\quad d_{\rm Kol}(F,Z) \leq 2   \left(   \frac{\E[ F^4 ] - 3}{p^2 + q^2 - 4pq} \right)^{1/2} \, ,
 \]
 see also Corollary 1.4 in \cite{DK17}.  
  \end{itemize}
 \end{rem}

  \subsection{Proof of  Theorem \ref{DK17-thm}}
 
Without losing any generality, we assume that $\Sigma_n = \Sigma$ and each component of $F^{(n)}$ belongs to $L^4(\P)$. Recall that
$
F^{(n)} = (F^{(n)}_{1}, \ldots, F^{(n)}_d)^T :=  \Big( Q_{q_1}\big(f_{1,n}; \mathbf{Y} \big),\ldots, Q_{q_d}\big(f_{d,n}; \mathbf{Y} \big) \Big)^T
$
 and we define 
 $
F^{(n)}_t = (F^{(n)}_{1,t}, \ldots, F^{(n)}_{d,t})^T
$ with $F^{(n)}_{i,t} :=   Q_{q_i}\big(f_{i,n}; \mathbf{Y}^t \big)$
 so that by Lemma \ref{exch} and Proposition \ref{GAMMA_exch}, $\big( F,  F_t  \big) : =  \big( F^{(n)},  F^{(n)}_t  \big)$ form an exchangeable pair satisfying the conditions in Proposition \ref{Meckes09} with  $\mathscr{G} = \sigma\{ \mathbf{X} \}$, $\Lambda = \text{diag}(q_1, \ldots, q_d)$ and
 \[
 S = \Big( 2 \Gamma\big(F^{(n)}_{i} , F^{(n)}_{j}  \big) - 2q_j \Sigma_{i,j}\Big)_{1\leq i,j\leq d}\,, \quad \rho_i\big(F^{(n)}\big) = -4q_i \,\E\Big[ ( F^{(n)}_i )^4\Big] + 12 \, \E\Big[ ( F^{(n)}_i )^2 \Gamma( F^{(n)}_i ,F^{(n)}_i ) \Big]  \, .
 \]
 Indeed, the condition (c) in  Proposition \ref{Meckes09} follows from the relation (c) in Proposition \ref{GAMMA_exch}, and for each $i,j\in\{1, \ldots, d\}$,  we have 
\[  \lim_{t\downarrow 0}  \frac{1}{t} \E\big[ F^{(n)}_{i,t} - F^{(n)}_i \vert \sigma\{ \mathbf{X} \} \big]  = - q_i F^{(n)}_i  \quad\text{in $L^4(\P)$,}
\]
and 
\[
  \lim_{t\downarrow 0}  \frac{1}{t} \E\Big[ \big(F^{(n)}_{i,t} - F^{(n)}_i \big)\big( F^{(n)}_{j,t} - F^{(n)}_j \big) \vert \sigma\{ \mathbf{X}\} \Big] = 2q_j \Sigma_{i,j} + \Big[  2 \Gamma\big(F^{(n)}_{i} , F^{(n)}_j \big) -  2q_j \Sigma_{i,j} \Big] \quad\text{in $L^2(\P)$.}
  \] 
 It follows that 
 \[
 \Big\| \frac{1}{t} \E\big[  F^{(n)}_{t} - F^{(n)} \vert \sigma\{\mathbf{X} \} \big] + \Lambda F^{(n)} \Big\| ^2_{\R^d} = \sum_{i=1}^d \left( \frac{1}{t} \E\big[ F^{(n)}_{i,t} - F^{(n)}_i \vert \sigma\{ \mathbf{X} \} \big]   + q_i F^{(n)}_i \right)^2
 \]
 converges to zero in $L^2(\P)$, as $t\downarrow 0$; and 
 \begin{align*}
&\quad \Big\|   \frac{1}{t} \E\Big[  \big( F^{(n)}_{t} - F^{(n)} \big)\big( F^{(n)}_{t} - F^{(n)} \big)^T   \vert \sigma\{\mathbf{X} \} \Big] - 2\Lambda \Sigma - S \Big\| ^2_{\rm H.S.} \\
& = \sum_{i,j=1}^d \left( \frac{1}{t} \E\Big[ \big(F^{(n)}_{i,t} - F^{(n)}_i \big)\big( F^{(n)}_{j,t} - F^{(n)}_j \big) \vert \sigma\{ \mathbf{X}\} \Big] -  2 \Gamma\big(F^{(n)}_{i} , F^{(n)}_j \big)  \right)^2 \quad\text{ converges to zero in $L^1(\P)$, as $t\downarrow 0$.}
 \end{align*}
  Hence we can apply Proposition \ref{Meckes09} and consequently, it suffices to show 
 \[
  \E\big[ \| S \| _{\rm H.S.} \big] + \sqrt{\sum_{i=1}^d \rho_i\big( F^{(n)} \big)  } \leq  \left( \sum_{i,j=1}^d \Var\Big( \Gamma\big(F^{(n)}_{i} , F^{(n)}_j \big) \Big) \right)^{1/2} +  \sqrt{\sum_{i=1}^d \rho_i\big( F^{(n)} \big)  }  \to 0 \,\,, \quad\text{as $n\to+\infty$.}
 \]
 In view of \eqref{use1} and  \eqref{use2}, it reduces to prove $\lim_{n\to+\infty} \Var\Big( \Gamma\big(F^{(n)}_{i} , F^{(n)}_j \big) \Big) = 0$ for $i < j$. We split this part into two steps.

 \paragraph{Step 1.} Suppose $F, G$ are two real random variables given as in Lemma \ref{DK-lemma} with $p\leq q$, then we have 
 \[
  \E\Big[ F^2 G^2 \Big] = \E[ FG ]^2 + \sum_{k=1}^{p+q-1} \Var\Big( J_k(FG) \Big) + (p+q)! \big\| f \widetilde{\otimes} g \mathbf{1}_{\Laplace_{p+q}} \big\| _{\H^{\otimes p+q}}^2
 \]
 and by \eqref{sp-eq} and Lemma \ref{tech1}, we get 
 \begin{align*}
&\qquad \frac{1}{q^2}  \Var\big( \Gamma(F,G) \big) \\
& \leq    \sum_{k=1}^{p+q-1}  \Var\big(  J_k(FG) \big) \leq  \Cov\big(F^2, G^2\big) - 2\E[ FG]^2 + (2q)!\sum_{r=1}^{p} r! {p\choose r}{q\choose r}      \min\Big\{ \,  \| f \|^2_{\H^{\otimes p}}   \mathcal{M}(g) ,   \| g \|^2_{\H^{\otimes q}}   \mathcal{M}(f)  \Big\} \, .
 \end{align*}
 Thus, we can further reduce our problem to show 
 \begin{align}\label{finally}
 \lim_{n\to+\infty} \Big(  \Cov\big(    (F^{(n)}_i)^2,   (F^{(n)}_j)^2\big) - 2\E[  F^{(n)}_i  F^{(n)}_j   ]^2 \Big) = 0 \quad\text{for any $1\leq i < j \leq d$,}
 \end{align}
 which will be carried out in the next step.

 \paragraph{Step 2.} Let $F, G$ be given as in previous step,   we have 
 \begin{align*}
 \E\big[F^2G^2 \big] &= \E\left( F^2 \left\{ \E[ G^2] + \sum_{k=1}^{2q-1} J_k(G^2) + J_{2q}(G^2) \right\} \right) \\
 &= \Var(F)\Var(G) + \E\left( F^2  \sum_{k=1}^{2q-1} J_k(G^2)  \right) + \mathbf{1}_{(p=q)} \E\big[ J_{2q}(F^2) J_{2q}(G^2) \big] \, .
 \end{align*}
 $\underline{\text{If $p < q$}}$, then $\E[ FG ] = 0$ and 
 \[
 \Big\vert   \Cov\big(   F^2 ,   G^2 \big) \Big\vert \leq \sqrt{ \E\big[ F^4 \big]} \sqrt{   \sum_{k=1}^{2q-1} \Var\big( J_k(G^2) \big)   } \leq   \sqrt{ \E\big[ F^4 \big]} \sqrt{   \E\big[ G^4 \big] - 3\E[G^2]^2 + \gamma_p\E[G^2] \mathcal{M}(g)  } \,\,,
 \]
where the second inequality follows from \eqref{use1} and the constant $\gamma_p$ is given therein. \\

\noindent{$\underline{\text{ If $p = q$}}$}, then  
\begin{align*}
\E\big[ J_{2q}(F^2) J_{2q}(G^2) \big] & = (2q)! \Big\langle f\widetilde{\otimes} f , g\widetilde{\otimes} g\mathbf{1}_{\Laplace_{2q}} \Big\rangle_{\H^{\otimes 2q}} = (2q)! \Big\langle f\widetilde{\otimes} f , g\widetilde{\otimes} g \Big\rangle_{\H^{\otimes 2q}} - (2q)! \Big\langle f\widetilde{\otimes} f , g\widetilde{\otimes} g\mathbf{1}_{\Laplace_{2q}^c} \Big\rangle_{\H^{\otimes 2q}} \\
& = 2q!^2 \langle f, g \rangle^2_{\H^{\otimes q}} + \sum_{r=1}^{q-1} q!^2 {q\choose r}^2 \big\langle f\otimes_r g , g\otimes_r f  \big\rangle_{\H^{\otimes 2q-2r}} - (2q)! \Big\langle f\widetilde{\otimes} f , g\widetilde{\otimes} g\mathbf{1}_{\Laplace_{2q}^c} \Big\rangle_{\H^{\otimes 2q}} \, ,
\end{align*}
 where the last equality follows from Lemma 2.2 in \cite{NR14}.   Consequently, $ \Cov\big(   F^2 ,   G^2 \big)  - 2 \E[FG]^2$ is equal to
 \begin{align}\label{3-terms}
   \E\left( F^2  \sum_{k=1}^{2q-1} J_k(G^2)  \right) + \sum_{r=1}^{q-1} q!^2 {q\choose r}^2 \big\langle f\otimes_r g , g\otimes_r f  \big\rangle_{\H^{\otimes 2q-2r}} - (2q)! \Big\langle f\widetilde{\otimes} f , g\widetilde{\otimes} g\mathbf{1}_{\Laplace_{2q}^c} \Big\rangle_{\H^{\otimes 2q}}   \, .
 \end{align}
 The first term in \eqref{3-terms} can be rewritten as ${\displaystyle \E \left[ \sum_{k=1}^{2q-1} J_k(F^2) J_k(G^2)\right]}$, which can be bounded by 
 \begin{align*}
  & \quad \sqrt{   \sum_{k=1}^{2q-1} \Var\big( J_k(F^2) \big)   }  \sqrt{   \sum_{k=1}^{2q-1} \Var\big( J_k(F^2) \big)   } \\
  &   \leq \sqrt{   \E\big[ F^4 \big] - 3\E[F^2]^2 + \gamma_q\E[F^2] \mathcal{M}(f)  }   \sqrt{   \E\big[ G^4 \big] - 3\E[G^2]^2 + \gamma_q\E[G^2] \mathcal{M}(g)  } \, \,;
 \end{align*}
 and the second term in  \eqref{3-terms} can be bounded by 
 \begin{align}
&\quad \sum_{r=1}^{q-1} q!^2 {q\choose r}^2 \big\| f\otimes_r g  \big\| _{\H^{\otimes 2q-2r}}^2  = \sum_{r=1}^{q-1} q!^2 {q\choose r}^2 \big\langle f\otimes_{q-r} f , g\otimes_{q-r} g \big\rangle_{\H^{\otimes 2r}}\label{fact-0} \\
 & \leq    \sum_{r=1}^{q-1} q!^2 {q\choose r}^2 \big\|  f\otimes_{q-r} f \big\| _{\H^{\otimes 2r}}   \cdot \big\| g\otimes_{q-r} g \big\| _{\H^{\otimes 2r}}    \notag \\
 & =  \sum_{r=1}^{q-1} q!^2 {q\choose r}^2 \big\|  f\otimes_r f \big\| _{\H^{\otimes 2q-2r}}   \cdot \big\| g\otimes_{r} g \big\| _{\H^{\otimes 2q-2r}}    \notag \\
 & \leq \sqrt{ \sum_{r=1}^{q-1} q!^2 {q\choose r}^2 \big\|  f\otimes_r f \big\| _{\H^{\otimes 2q-2r}}^2  } \sqrt{ \sum_{r=1}^{q-1} q!^2 {q\choose r}^2 \big\|  g\otimes_r g \big\| _{\H^{\otimes 2q-2r}}^2  } \label{CS-1} \\
 & \leq  \sqrt{ \E\big[ F^4 \big]- 3\E[F^2]^2 +  \gamma_q \mathcal{M}(f)\E[ F^2]   } \sqrt{ \E\big[ G^4 \big]- 3\E[G^2]^2 +  \gamma_q \mathcal{M}(g)\E[G^2]   }\,\,, \label{hi}
 \end{align}
where  \eqref{fact-0} follows from the easy fact that $ \| f\otimes_r g   \| _{\H^{\otimes 2q-2r}}^2  =  \big\langle f\otimes_{q-r} f , g\otimes_{q-r} g \big\rangle_{\H^{\otimes 2r}}$, and we used Cauchy-Schwarz inequality in \eqref{CS-1}, while \eqref{hi} can be deduced from Lemma \ref{tech1} and \eqref{use1}; finally, the third term in  \eqref{3-terms} can be bounded by 
 $
 \| f \| _{\H^{\otimes q}}^2 (2q)! \big\| g\widetilde{\otimes} g\mathbf{1}_{\Laplace_{2q}^c} \big\| _{\H^{\otimes 2q}} \leq   \| f \| _{\H^{\otimes q}}^2  \sqrt{(2q)! \gamma_q \E[G^2]  \mathcal{M}(g)}
 $.
To conclude this case, we obtain  
\begin{align*}
 \big\vert  \Cov\big(   F^2 ,   G^2 \big)  - 2 \E[FG]^2 \big\vert  \leq &\,\, 2\sqrt{ \Big(\E\big[ F^4 \big]- 3\E[F^2]^2 +  \gamma_q \mathcal{M}(f)\E[ F^2]  \Big)\Big( \E\big[ G^4 \big]- 3\E[G^2]^2   +  \gamma_q \mathcal{M}(g)\E[G^2]\Big)   } \\
 & +  \| f \| _{\H^{\otimes q}}^2  \sqrt{(2q)! \gamma_q \E[G^2]  \mathcal{M}(g)} \, . 
\end{align*}
Combining the above two cases, we get immediately the relation  \eqref{finally}, and hence we finish the proof of Theorem \ref{PT-Rad}.

 \subsection{Proofs of technical lemmas}
 
\paragraph{Proof of Lemma \ref{DK-lemma}} Let us first introduce some notation:  if $F = \mathfrak{f}\big(\mathbf{X})$, we write 
  \begin{center} 
  $F^{\oplus k} = \mathfrak{f}\big(X_1, \ldots, X_{k-1}, +1, X_{k+1}, \ldots \big)$ and $F^{\ominus k} = \mathfrak{f}\big(X_1, \ldots, X_{k-1}, -1, X_{k+1}, \ldots \big)$,
  \end{center}
   we define the {\it discrete gradient} $D_k F = \sqrt{p_kq_k}\big( F^{\oplus k} -  F^{\ominus k}\big)$, in particular, $D_k Y_k = 1$. We can define the iterated gradients $D^{(m)}_{k_1, \ldots, k_m} = D_{k_1}\circ D^{(m-1)}_{k_2, \ldots, k_m}$ with $D^{(1)}_k = D_k$. For example,  $D_k Q_d(f) = d Q_{d-1}\big( f(k,\cdot) \big)$ and $D^{(2)}_{k,\ell} Q_d(f) = d(d-1) Q_{d-1}\big( f(k,\ell,\cdot) \big)$ for $d\geq 2$ and $f\in\H^{\odot d}_0$, see \cite{KRT_aop} for more details.
\begin{proof}
It is clear that $FG\in L^2(\P)$ has the  chaotic expansion
\[
FG = \E[ FG ] + \sum_{m\geq 1} Q_m(h_m)\,,
\]
where for each $m\in\N$, the kernel $h_m\in\H^{\odot m}_0$ is given by $h_m(k_1, \ldots, k_m) : = \frac{1}{m!} \E\Big[ D^{(m)}_{k_1,\ldots, k_m} (FG) \Big]$, due to the Stroock's formula  (Proposition 2.1 in \cite{KRT_aop}).    So it suffices to show that 
\begin{align}\label{(1)}
D^{(p+q)}_{k_1, \ldots, k_{p+q}} (FG) = (p+q)! (f\widetilde{\otimes} g)(k_1, \ldots, k_{p+q})\mathbf{1}_{\Laplace_{p+q}}(k_1, \ldots, k_{p+q}) \quad{\rm and}\quad  D^{(s)}_{k_1, \ldots, k_s} (FG) = 0 
\end{align}
for any $s > p+q$. Note that the second part follows immediately from  the first one.    

  Recall the product formula (see \emph{e.g.} \cite[(2.4)]{KRT_aop}) for the discrete gradient $D_k$: for $F,G\in L^2(\P)$, 
\[
D_k(FG) = (D_kF)G + F(D_kG)  - \frac{X_k}{\sqrt{p_kq_k}} (D_kF)(D_kG) = : D_k^L(FG) + D_k^R(FG) + D_k^M(FG) \,\,,
\] 
that is, we decompose $D_k$ into three operations $D_k^L$, $D_k^R$ and $D_k^M$.   Therefore, we can write for $k_1 <  \ldots < k_{p+q}$,
  \begin{align}
D^{(p+q)}_{k_1, \ldots, k_{p+q}} (FG) & =  \sum_{A_1, \ldots, A_{p+q}\in\{ L, M, R \} } D^{A_1}_{k_1}\circ \cdots \circ D^{A_{p+q}}_{k_{p+q}} (FG)  =  \sum_{A_1, \ldots, A_{p+q}\in\{ L, R \} } D^{A_1}_{k_1}\circ \cdots \circ D^{A_{p+q}}_{k_{p+q}} (FG) \, , \notag
  \end{align}
  where the last equality follows from the fact that for $k\neq \ell$, $D_\ell (X_k F) = X_k D_\ell F$.  Moreover,  $D^{A_1}_{k_1}\circ \cdots \circ D^{A_{p+q}}_{k_{p+q}} (FG) = 0$ unless  $L$ appears exactly  $p$ times and $R$ appears exactly $q$ times  in the words $A_1, \ldots, A_{p+q}$, so that   one can further rewrite $D^{(p+q)}_{k_1, \ldots, k_{p+q}} (FG) $   as
 \begin{align*} 
 &\qquad  \sum_{\substack{\sigma\in\mathfrak{S}_{p+q}:  \\ \sigma(1) < \ldots < \sigma(p) \\ \sigma(p+1) < \ldots < \sigma(p+q)   }} \Big( D^{(p)}_{k_{\sigma(1)} , \ldots, k_{\sigma(p)}   }F \Big)\Big( D^{(q)}_{k_{\sigma(p+1)} , \ldots, k_{\sigma(p+q)}   }G \Big)  =\sum_{\sigma\in\mathfrak{S}_{p+q}}  f\big(k_{\sigma(1)} , \ldots, k_{\sigma(p)} \big) g\big(k_{\sigma(p+1)} , \ldots, k_{\sigma(p+q)} \big)\,, 
 \end{align*}
where the last equality follows from the symmetry of $f$ and $g$,  and it gives us  $ D^{(p+q)}_{k_1, \ldots, k_{p+q}} (FG)  = (p+q)! (f\widetilde{\otimes} g)(k_1, \ldots, k_{p+q})$. This proves \eqref{(1)}, while the particular case follows from again the Stroock's formula. More precisely,  one can first deduce from the previous discussion that 
 $
 Q_1(h)^2 = \| h \| ^2 _\H + Q_1(w) + Q_2\big(h\otimes h \mathbf{1}_{\Laplace_2}\big)
 $
 for some $w\in\H$ given by $w(k) : = \E\big[ D_k\big( Q_1(h)^2\big) \big]$. By the definition of discrete gradient, one has
 \begin{align*}
D_k\big( Q_1(h)^2\big) &= \sqrt{p_kq_k} \left\{\, \left( \sum_{j\neq k} h(j)Y_j + h(k) \frac{1 - p_k +q_k}{2 \sqrt{p_kq_k}   } \right)^2 -  \left( \sum_{j\neq k} h(j)Y_j + h(k) \frac{-1 - p_k +q_k}{2 \sqrt{p_kq_k}   } \right)^2\, \right\} \\
& = h(k)^2 \frac{q_k - p_k}{\sqrt{p_kq_k}} + 2h(k) \sum_{j\neq k} h(j) Y_j \,\,,
 \end{align*}
 which concludes our proof of  Lemma \ref{DK-lemma}. \qedhere

\end{proof}

\paragraph{Proof of Lemma \ref{tech1}:}  It follows from Lemma \ref{DK-lemma} that $FG = \E[ FG ] +{\displaystyle  \sum_{k=1}^{p+q-1} } J_k(FG) + Q_{p+q}\Big(  f\widetilde{\otimes} g \mathbf{1}_{\Laplace_{p+q}} \Big)$, therefore by orthogonality property, one has
\begin{align*}
 \E\big[ F^2G^2 \big]  & = \E[ FG]^2 +  \sum_{k=1}^{p+q-1} \Var\big( J_k(FG) \big) + (p+q)! \, \big\| f\widetilde{\otimes} g \mathbf{1}_{\Laplace_{p+q}} \big\|_{\H^{\otimes p+q}}^2 \\
 & = \E[ FG]^2 +  \sum_{k=1}^{p+q-1} \Var\big( J_k(FG) \big) + (p+q)! \, \big\| f\widetilde{\otimes} g  \big\|_{\H^{\otimes p+q}}^2 -  (p+q)!  \big\| f\widetilde{\otimes} g \mathbf{1}_{\Laplace_{p+q}^c} \big\|_{\H^{\otimes p+q}}^2 \, .
\end{align*}
Recall from \cite[Lemma 2.2]{NR14} that 
\begin{align}\label{NR14-0}
(p+q)! \, \big\| f\widetilde{\otimes} g  \big\|_{\H^{\otimes p+q}}^2 = p!q!\, \sum_{r = 0}^{p\wedge q} {p\choose r}{q\choose r} \big\| f\otimes_r g \big\|_{\H^{\otimes p+q-2r}}^2 \geq p!q! \| f \| _{\H^{\otimes p}}^2 \| g\| _{\H^{\otimes q}}^2 + \mathbf{1}_{(p=q)} p!^2 \big\langle f,  g \big\rangle_{\H^{\otimes p}}^2 \, ,
\end{align}
thus \eqref{lem-1} follows by noticing that $\E[ FG ] =  \mathbf{1}_{(p=q)} p! \big\langle f,  g \big\rangle_{\H^{\otimes p}}$ and $\Var(F)\Var(G) = p!q! \| f \| _{\H^{\otimes p}}^2 \| g\| _{\H^{\otimes q}}^2$.
 
 Using \eqref{NR14-0} again, we have 
 \begin{align} \label{uni}
  \sum_{k=1}^{p+q-1} \Var\big( J_k(F^2) \big) =  \E\big[F^4\big] - 3\E[ F^2]^2 - p!^2  \sum_{r = 1}^{ p-1 } {p\choose r}^2 \big\| f\otimes_r f \big\|_{\H^{\otimes 2p-2r}}^2 +  (2p)! \big\| f \widetilde{\otimes} f \mathbf{1}_{\Laplace_{2p}^c} \big\| ^2_{\H^{\otimes 2p}} \,\,,
 \end{align}
 which implies \eqref{lem-2}.
 
 It remains to prove \eqref{lem-3} and  we'll  use the same  arguments as in the proof of \cite[Lemma 3.3]{DK17}: 
 \begin{align}
 \big\| f \widetilde{\otimes} g \mathbf{1}_{\Laplace_{p+q}^c} \big\| ^2_{\H^{\otimes p+q}} & \leq \big\| f \otimes g \mathbf{1}_{\Laplace_{p+q}^c} \big\| ^2_{\H^{\otimes p+q}} = \sum_{(i_1, \ldots, i_p, j_1, \ldots, j_q)\in \Laplace_{p+q}^c} f(i_1, \ldots, i_p)^2 g(j_1, \ldots, j_q)^2 \notag \\
 & = \sum_{r=1}^{p\wedge q} r! {p\choose r}{q\choose r}    \sum_{\substack{  (i_1, \ldots, i_p)\in \Laplace_p \\   (j_1, \ldots, j_q)\in \Laplace_q \\ \text{card}(\{ i_1, \ldots, i_p\} \cap \{ j_1, \ldots, j_q\}) =r }} f(i_1, \ldots, i_p)^2 g(j_1, \ldots, j_q)^2  \,\,, \label{iner-sum}
  \end{align}
 where ${\rm card}(A)$ means the cardinality of the set $A$, and the combinatorial constant $ r! {p\choose r}{q\choose r} $ is the number of ways one can build $r$ pairs of identical indices out of  $(i_1, \ldots, i_p)\in \Laplace_p$ and   $(j_1, \ldots, j_q)\in \Laplace_q$.  
 
 Therefore, it is enough to notice that for each $r\in\{1, \ldots, p\wedge q\}$, the inner sum in \eqref{iner-sum} is bounded by
  \begin{align*}
&  \qquad   \sum_{\substack{  (i_1, \ldots, i_{p-r}, k_1, \ldots, k_r)\in \Laplace_p \\   (j_1, \ldots, j_{q-r}, k_1, \ldots, k_r)\in \Laplace_q  }} f(i_1, \ldots, i_{p-r}, k_1, \ldots, k_r)^2 g(j_1, \ldots, j_{q-r}, k_1, \ldots, k_r)^2 \\
 &\leq   \sum_{\substack{  (i_1, \ldots, i_{p-1}, k)\in \Laplace_p \\   (j_1, \ldots, j_{q-1}, k)\in \Laplace_q  }} f(i_1, \ldots, i_{p-1}, k )^2 g(j_1, \ldots, j_{q-1}, k )^2   \leq        \min\Big\{ \,  \| f \|^2_{\H^{\otimes p}}   \mathcal{M}(g) ,   \| g \|^2_{\H^{\otimes q}}   \mathcal{M}(f)  \Big\} \,\, .
 \end{align*}
 The proof of Lemma \ref{tech1} is complete.

\section{Universality of Homogeneous sums}

Fix $d\geq 2$ and a {\it divergent} sequence $(N_n, n\geq 1)$ of natural numbers.   Consider the kernels $f_n:\{1, \ldots, N_n \}^d\to \R$ symmetric and vanishing on diagonals and $d!\| f_n \| ^2_{\H^{\otimes d}}=1$, then according to \eqref{Q_d},
\[
Q_d(f_n; \Xi) = \sum_{i_1, \ldots, i_d\leq N_n} f_n(i_1, \ldots, i_d) \xi_{i_1}\cdots \xi_{i_d} \, .
\]
The following central limit theorem due to de Jong \cite{deJong90} gave sufficient conditions for asymptotic normality of $Q_d(f_n; \Xi)$.

\begin{thm} Under the above setting, let  $\Xi = (\xi_i, i\geq 1)$ be   a sequence of independent centred random variables with unit variance and {\it finite fourth moments}.   If $\E\big[ Q_d(f_n; \Xi)^4 \big]\to 3$ and the maximal influence $\mathcal{M}(f_n)\to 0$ as $n\to+\infty$, then $Q_d(f_n; \Xi)$ converges in law to a standard Gaussian.  
\end{thm}
The above result exhibits the universality phenomenon as well as the importance of the notion ``maximal influence".  Another striking result with similar nature is the invariance principle established in \cite{MOO10}, in which the authors were able to control distributional distance between homogeneous sums over different sequences of independent random variables in terms of  maximal influence, see \emph{e.g.} Theorem 2.1 therein.  

 Let us restrict ourselves to the Gaussian setting for a while: when $\mathbf{G}$ is a sequence of {\it i.i.d.} standard Gaussians, $Q_d(f_n;\mathbf{G})$ belongs to the $d$-th Gaussian Wiener chaos, and  the fourth moment theorem \cite{FMT} implies that if  $Q_d(f_n;\mathbf{G})$ converges in law to a standard Gaussian (or equivalently $\E\big[ Q_d(f_n; \mathbf{G})^4 \big]\to 3$), then $\| f_n \otimes_{d-1} f_n \| _{\H^{\otimes 2}}\to 0$. While $\mathcal{M}(f_n) \leq \| f_n\otimes_{d-1} f_n \| _{\H^{\otimes 2}}$  due to  \cite[Lemma 2.4]{NPR-rad},  so that $\mathcal{M}(f_n)\to 0$.  This  hints the universality of the Gaussian Wiener chaos, see \cite{NPR_aop} for more details. 
    
 The following result is (slightly) adapted from Theorem 7.5 in  \cite{NPR_aop}.
 
 \begin{thm}\label{Univer}  Fix integers $d\geq 2$ and $q_d\geq \ldots \geq q_1\geq 2$. For each $j\in\{1,\dotsc,d\}$, let $(N_{j,n}, n\geq 1)$ be a sequence of natural numbers diverging to infinity, and  let $f_{j,n}: \{1, \ldots, N_{j,n  } \} ^{q_j}\to \R$ be symmetric and vanishing on diagonals (\emph{i.e.} $f_{j,n}\in\H^{\odot q_j}_0$ with support contained in $\{ 1, \ldots, N_{j,n} \}^{q_j}$) such that 
\[
\lim_{n\to+\infty}\mathbf{1}_{(q_k = q_l)} q_k! \sum_{i_1, \ldots, i_{q_k} \leq N_{k,n}  }  f_{k, n}(i_1, \ldots, i_{q_k})  f_{l,n}(i_1, \ldots, i_{q_k})  = \Sigma_{k,l}\, ,
\]
where $\Sigma = (\Sigma_{i,j}, 1\leq i,j\leq d)$ is a symmetric nonnegative definite $d$ by $d$ matrix.  Then the following  statements are equivalent:

\begin{enumerate}

\item[$(A_1)$] Given  a sequence  $\mathbf{G}$  of i.i.d. standard Gaussians,
$
\big(  Q_{q_1} (f_{1,n}; \mathbf{G} ), \ldots, Q_{q_d} (f_{d,n}; \mathbf{G} ) \big)^T
$
 converges in distribution to $\mathcal{N}(0, \Sigma)$, as $n\to+\infty$.

\item[$(A_2)$] For every sequence $\Xi= \big( \xi_i, i\in\N\big)$ of independent centred random variables with unit variance and $\sup_{i\in\N} \E\big[ \vert \xi_i\vert^3 \big] < +\infty$, the sequence of $d$-dimensional random vectors  $\big(  Q_{q_1} (f_{1,n} ;\Xi ), \ldots, Q_{q_d} (f_{d,n}; \Xi ) \big)^T$
converges in distribution to $\mathcal{N}(0, \Sigma)$, as $n\to+\infty$.
\end{enumerate}
  \end{thm}
 Similar universality result for Poisson chaos was first established in \cite{PZheng14} and refined recently in \cite{DVZ17}. It was pointed out in \cite{PZheng14} and \cite{NPR_aop} that homogeneous sums inside the  Rademacher chaos are not universal with respect to normal approximation and a counterexample is available \emph{e.g.} in \cite[Proposition 1.7]{PZheng14}:
 
\paragraph{A Counterexample:}  Let $\mathbf{Y}$ be a sequence of \emph{i.i.d.} random variables with $\P(Y_1 = 1) = \P(Y_1 = - 1) = 1/2$ (that is, in the symmetric setting).
  Fix $q\geq 2$ and for each  $N \geq q$, we set 
 \begin{align*}
 f_N(i_1, \ldots, i_q)  = \begin{cases}
 \dfrac{1}{q! \sqrt{N-q+1}} \,, \quad\text{if $\{ i_1, \ldots, i_q\} = \{ 1, 2, \ldots, q-1, s\}$ for $q\leq s \leq N$;} \\
\qquad 0 \,, \qquad\qquad\text{otherwise.} 
 \end{cases}
 \end{align*}
 Then in the symmetric case,  
 $$Q_q(f_N;\mathbf{Y}) = Y_1Y_2\cdots Y_{q-1} \sum_{i=q}^N \frac{Y_i}{\sqrt{N-q+1}}   $$
  converges in law to the standard Gaussian, while if $\mathbf{G}$  is a sequence of i.i.d. standard Gaussians, then for every $N\geq 2$, $Q_q(f_N;\mathbf{G})\overset{law}{=} G_1G_2\cdots G_q$  fails to be Gaussian.   It is easy to check that the maximal influence $\mathcal{M}(f_N)$ of the kernel $f_N$ is equal to $1/(qq!)$ for every $N\geq 2$, which is consistent with  de Jong's theorem.

In the end of this section, we provide a (partially) universal result   for Rademacher chaos that complements \cite{DVZ17, NPR_aop, PZheng14}.

\begin{prop} Let the assumptions in Theorem \ref{Univer} prevail. Then, the following  statement is equivalent to $(A_1)$ and $(A_2)$  in Theorem \ref{Univer}:
\begin{enumerate}
\item[$(A_3)$] in the symmetric case, as $n\to+\infty$,
$
\big(  Q_{q_1} (f_{1,n}; \mathbf{Y} ), \ldots, Q_{q_d} (f_{d,n}; \mathbf{Y} ) \big)^T
$
 converges in distribution to $\mathcal{N}(0, \Sigma)$, and $\mathcal{M}(f_{j,n})\to 0$ for each $j\in\{1, \ldots, d\}$.
\end{enumerate}

\end{prop}
 \begin{proof}  Suppose $(A_1)$ holds true, then $
\big(  Q_{q_1} (f_{1,n}; \mathbf{Y} ), \ldots, Q_{q_d} (f_{d,n}; \mathbf{Y} ) \big)^T
$
 converges in distribution to $\mathcal{N}(0, \Sigma)$ by ``$(A_2)\Leftrightarrow (A_1)$"; and by the fourth moment theorem on a Gaussian space \cite{FMT}, $(A_1)$ implies that $\| f_{j,n}\otimes_{q_j-1} f_{j,n} \| _{\H^{\otimes 2}}\to 0$, as $n\to+\infty$. Recall from \cite[Lemma 2.4]{NPR-rad} that   $\mathcal{M}(f) \leq \| f\otimes_{d-1} f \| _{\H^{\otimes 2}}$ for each $f\in\H^{\odot d}_0$, therefore $\mathcal{M}(f_{j,n})\to 0$ for each $j\in\{1, \ldots, d\}$. This proves the implication ``$(A_1)\Rightarrow (A_3)$".
 
 It remains to show ``$(A_3)\Rightarrow (A_1)$". Now we  assume that $(A_3)$ is true, then by a weak form of the \emph{hypercontractivity} property (see Section 2), we have 
 $
\lim_{n\to+\infty} \E\big[ Q_{q_j} (f_{n,j}; \mathbf{Y} )^4  \big] = 3 \Sigma_{j,j}^2 
 $  for each $j=1,\ldots, d$.
 It follows from Lemma \ref{tech1} that $\big\| f_{j,n}\otimes_r f_{j,n} \big\| _{\H^{\otimes 2q_j - 2r}} \to 0$ for each $r = 1, \ldots, q_j - 1$, and any $j=1, \ldots, d$. Hence, $(A_1)$ follows immediately from the Peccati-Tudor theorem \cite{PTudor}. This concludes our proof. \qedhere

 \end{proof}

\end{document}